\documentclass[12pt,reqno]{amsart}
\usepackage[notref,notcite]{}
\usepackage{amssymb}
\usepackage{amsmath}
\advance\hoffset by -0.5 truecm

\numberwithin{equation}{section}
  \theoremstyle{plain}
  \newtheorem{thm}{Theorem}[section]
  \theoremstyle{plain}
  \newtheorem{lem}[thm]{Lemma}
  \theoremstyle{plain}
  \newtheorem{prop}[thm]{Proposition}
  \theoremstyle{plain}

\theoremstyle{remark}

  \newcommand{\eps}{\varepsilon}
  
  \newcommand{\R}{\mathbb R}

\begin{document}

\def\binom#1#2{{#1\choose #2}}

\title{Tomography of small residual stresses}
\thanks{This work was partially supported by a 2008-2010 NSC-RFBR Bilateral Research
Project. The first author was also partially supported by the RFBR Grant 11-01-12106-ofi-m-2011. The authors are grateful to NSC, Taiwan and RFBR,
Russia for the financial support.}
\author{Vladimir Sharafutdinov and Jenn-Nan Wang}
\address{Sobolev Institute of Mathematics, Russia}
\email{sharaf@math.nsc.ru}
\thanks{}
\address{Department of Mathematics, National Taiwan University, Taipei 106, Taiwan}
\email{jnwang@math.ntu.edu.tw}
\thanks{}


\maketitle

\begin{abstract}
We study the inverse problem of determining the
residual stress in Man's model using tomographic data.
Theoretically, the tomographic data are obtained at the zero
approximation of geometrical optics for Man's residual stress model.
For compressional waves, the inverse problem is equivalent to the
problem of inverting the longitudinal ray transform of a symmetric
tensor field. For shear waves, the inverse problem, after the
linearization, leads to another integral geometry operator which is
called the mixed ray transform. Under some restrictions on
coefficients, we are able to prove the uniqueness results in these
two cases.
\end{abstract}



\section[]{Introduction}

Residual stresses are the stresses existing within an elastic body
in the absence of external loads. They can be caused, for example,
by heat treatment. The (Cauchy) residual stress tensor $\tilde R$
is symmetric, ${\tilde R}_{ij}={\tilde R}_{ji}$, satisfies the equilibrium equation
\begin{equation}\label{1}
\text{div}\tilde R=0\quad \text{in}\quad M
\end{equation}
and the traction-free boundary condition
\begin{equation}\label{1'}
\tilde RN=0\quad \text{on}\quad
\partial M,
\end{equation}
where $M\subset\R^3$ is the domain occupied by an elastic body and
$N$ is the normal to the boundary $\partial M$. Uncontrolled
residual stresses are often detrimental to the life of a structural
component. Nevertheless, in some cases, residual stresses are
introduced on purpose to make materials or structures more resistant
to damage, for example, toughened glass. It is therefore desirable
to design reliable methods to determine the residual stress in a
body. Besides its practical value, the determination of residual
stresses gives rise to many challenging mathematical questions as
well. Based on different model equations for the residual stress and
using different types of measurements, several versions of the
inverse problem of unique determination of the residual stress were
studied in \cite{hu}, \cite{ho}, \cite{iwy1}, \cite{iwy2},
\cite{imn}, \cite{ra} (see also \cite{ro1}, \cite{ro2} for related
results).

In this paper, we study the problem of recovering the residual stress tensor from
the viewpoint of tensor tomography. To this end, we
consider the time-stationary elastic system with residual stress that is assumed to be
comparable with the
inverse of the frequency. To be more precise, we set $\tilde R=R/\omega$, where $\omega\in\R$ is the (angular) frequency of an elastic wave. In what follows,
we refer to $R$ as the residual stress tensor again although its physical dimension is time$\cdot$stress = time$\cdot$force/area.
The time-stationary elastic wave equation is
\begin{equation}\label{man0}
\text{div}\sigma+\omega^2\rho u=0,
\end{equation}
where  $u$ is the displacement
vector, $\rho>0$ is the density and $\sigma$ is the first
Piola-Kirchhoff stress tensor. In view of Man's model \cite{Man}, the first Piola-Kirchhoff stress tensor is expressed through the displacement vector by the equation
\begin{equation}  \label{Man1}
\begin{aligned}
\sigma&=\lambda(\mbox{tr}\,\varepsilon)I+2\mu\varepsilon+\frac{1}{\omega}\Big{[}R+{\nabla}u\cdot R\\
&+\nu_1(\mbox{tr}\,\varepsilon)(\mbox{tr}\,R)I
+\nu_2(\mbox{tr}\,R)\varepsilon
+\nu_3\Big((\mbox{tr}\,\varepsilon)R+(\mbox{tr}\,\varepsilon
R)I\Big) +\nu_4(\varepsilon R+R\varepsilon)\Big{]},
\end{aligned}
\end{equation}
where  $\eps$ is the infinitesimal strain tensor, $\lambda$ and $\mu$
are Lam\'e's parameters, and $\nu_1,\dots,\nu_4$ are Man's
parameters. Ignoring the term with $1/\omega$ in (\ref{Man1}), we
obtain the standard Hooke's law
\begin{equation}  \label{BG}
\sigma=\lambda(\mbox{tr}\,\varepsilon)I+2\mu\varepsilon.
\end{equation}
We refer to (\ref{BG}) as the {\it background isotropic medium} for
(\ref{Man1}), while (\ref{Man1}) is called the {\it quasi-isotropic
perturbation} of the medium (\ref{BG}) caused by the residual stress
tensor $R$. Here, the term ``quasi-isotropic'' is used to emphasize
that the perturbation smallness parameter $1/\omega$ coincides with
the reciprocal of the frequency. The concept of the quasi-isotropic
approximation or quasi-isotropic media was first introduced by
Kravtsov \cite{Kr} for the Maxwell equations (see also \cite{KO}).
This approximation is based on the method of geometrical optics with
rays generated by the background isotropic medium. Its
generalization to elastic waves is presented in \cite[Ch. 7]{mb}.

Similar to the method used in \cite{mb} for a general anisotropic
perturbation of an isotropic elastic medium, we start with applying
the classical ray method to \eqref{man0} with the stress tensor given by
\eqref{Man1}. We restrict ourselves to considering the zero
approximation of geometrical optics. As compared with the classical
case of isotropic media, our formulas for the zero approximation
have two distinct features. First, an additional factor appears in
the formula for the amplitude of a compressional wave to describe
the accumulation of the wave phase along a ray due to the residual
stress. Second, the Rytov law for shear waves has got an additional
term depending linearly on the residual stress. The inverse problem
we study in this work is to determine the residual stress from the
results of registration of the compressional or shear waves on the
boundary of the medium under investigation. For compressional waves,
the inverse problem is equivalent to the problem of inverting the
longitudinal ray transform of a symmetric tensor field. For shear
waves, the inverse problem, after the linearization, leads to
another integral geometry operator that is called the mixed ray
transform.

The paper is organized as follows. In Section~2, we discuss the
quasi-isotropic approximation for Man's model of residual stress.
For compressional waves, we derive the formula for the amplitude; while for shear waves, we obtain the
Rytov law. All the content of Section~2 is actually some modification of the corresponding arguments from \cite[Section 7.1]{mb}. We have chosen the following compromise presentation style in Section~2: all details are presented for compressional waves while our arguments for shear waves are condensed as far as possible.
In Section~3, we investigate the inverse problem of
determining the residual stress from measurements of compressional
waves on the boundary. The inverse problem of determining the
residual stress using shear waves is discussed in Section~4.

\section[]{Quasi-isotropic approximation for residual stresses}

In this section, we derive the quasi-isotropic approximation for the
residual stress model \eqref{Man1} in detail. We first write (\ref{Man1}) in the form
\begin{equation}  \label{Man2}
\sigma=\lambda(\mbox{tr}\,\varepsilon)I+2\mu\varepsilon+\frac{1}{\omega}(R+{\nabla}u\cdot
R+c\varepsilon),
\end{equation}
where $c=c(R)$ is a fourth rank tensor depending linearly on $R$. To write down the dependence explicitly, we reproduce (\ref{Man1}) and (\ref{Man2}) in Cartesian coordinates
\begin{equation}  \label{Man3}
\begin{aligned}
\sigma_{jk}&=\lambda\varepsilon_{pp}\delta_{jk}+2\mu\varepsilon_{jk}+\frac{1}{\omega}\Big{(}R_{jk}+{\nabla}_{\!j}u_p\cdot R_{pk}\\
&+\nu_1 R_{pp}\varepsilon_{qq}\delta_{jk} +\nu_2
R_{pp}\varepsilon_{jk}
+\nu_3(\varepsilon_{pp}R_{jk}+\varepsilon_{pq} R_{qp}\delta_{jk})
+\nu_4(\varepsilon_{jp} R_{pk}+R_{jp}\varepsilon_{pk})\Big{)}
\end{aligned}
\end{equation}
and
\begin{equation}  \label{Man4}
\sigma_{jk}=\lambda\varepsilon_{pp}\delta_{jk}+2\mu\varepsilon_{jk}+\frac{1}{\omega}(R_{jk}+{\nabla}_{\!j}u_p\cdot
R_{pk}+c_{jklm}\varepsilon_{lm}).
\end{equation}
Comparing these two formulas, we deduce
\begin{equation}  \label{1.3}
\begin{aligned}
c_{jklm}&=\nu_1 R_{pp}\delta_{jk}\delta_{lm}
+\frac{\nu_2}{2} R_{pp}(\delta_{jl}\delta_{km}+\delta_{jm}\delta_{kl})\\
+&\nu_3(R_{jk}\delta_{lm}+R_{lm}\delta_{jk})
+\frac{\nu_4}{2}(R_{jl}\delta_{km}+R_{jm}\delta_{kl}+R_{kl}\delta_{jm}+R_{km}\delta_{jl}).
\end{aligned}
\end{equation}

Since we are going to use curvilinear coordinates, we rewrite
(\ref{1.3}) and (\ref{Man4}) in the covariant form
\begin{equation}  \label{1.4}
\begin{aligned}
c_{jklm}&=\nu_1 (\mbox{tr}_ER) g_{jk}g_{lm}
+\frac{\nu_2}{2} (\mbox{tr}_ER)(g_{jl}g_{km}+g_{jm}g_{kl})\\
+&\nu_3(R_{jk}g_{lm}+R_{lm}g_{jk})
+\frac{\nu_4}{2}(R_{jl}g_{km}+R_{jm}g_{kl}+R_{kl}g_{jm}+R_{km}g_{jl}),
\end{aligned}
\end{equation}
\begin{equation}  \label{1.5}
\sigma_{jk}=\lambda\varepsilon^p_p
g_{jk}+2\mu\varepsilon_{jk}+\frac{1}{\omega}(R_{jk}+u^p_{\! \
;j}R_{pk}+c_{jklm}\varepsilon^{lm}).
\end{equation}
Here $(g_{jk})$ is the Euclidean metric tensor such that
$|dx|^2_E=g_{jk}dx^jdx^k$, $(\mbox{tr}_ER)=g^{jk}R_{jk}$,
$\varepsilon^{lm}=g^{jl}g^{km}\varepsilon_{jk}$, and
$\eps_m^m=g^{lm}\eps_{lm}$. Recall that $\eps_{lm}=(u_{l;\, m}+u_{m;\,
l})/2$, $u_l=g_{lk}u^k$,
$$u_{l;\> m}=\frac{\partial u_l}{\partial
x^m}-\Gamma_{lm}^ku_k\quad\mbox{and}\quad
u^l_{\> ;\> m}=\frac{\partial u^l}{\partial
x^m}+\Gamma_{km}^lu^k,$$ where $\Gamma_{km}^l$ are the Christoffel
symbol given by
$$\Gamma_{km}^l=\frac 12g^{lj}\Big(\frac{\partial g_{kj}}{\partial
x^m}+\frac{\partial g_{mj}}{\partial x^k}-\frac{\partial
g_{km}}{\partial x^j}\Big).$$
In curvilinear coordinates, \eqref{man0}
is expressed as
\begin{equation}\label{man5}
\sigma_{jk;}{}^k+\omega^2\rho u_j=0,
\end{equation}
where $\sigma_{jk;}{}^l=g^{lm}\sigma_{jk;\> m}$ and
$$\sigma_{jk;\> m}=\frac{\partial\sigma_{jk}}{\partial
x^m}-\Gamma_{jm}^l\sigma_{lk}-\Gamma_{km}^l\sigma_{lj}.$$ It is
clear that the tensor $c$ possesses the symmetries
\begin{equation}  \label{1.6}
c_{jklm}=c_{kjlm}=c_{jkml}=c_{lmjk}
\end{equation}
as follows from (\ref{1.4}). The equilibrium
equation \eqref{1} is now written as
\begin{equation}  \label{1.7}
{R_{jk\>;}}^k=0.
\end{equation}

We consider propagation of small elastic waves in a medium described
by \eqref{man5} with the constitutive law (\ref{1.5}). We reproduce
some arguments of book \cite{mb}. The method
of geometrical optics consists of representing a solution to the
system by the asymptotic series
$$
u_ j=e^{i\omega\tau}\sum\limits_{m=0}^\infty\frac{{{\stackrel m u}}_j}{(i\omega)^m},\quad
\varepsilon_{jk}=e^{i\omega\tau}\sum\limits_{m=-1}^\infty\frac{{\stackrel m \varepsilon}_{jk}}{(i\omega)^m},\quad
\sigma_{jk}=e^{i\omega\tau}\sum\limits_{m=-1}^\infty\frac{{\stackrel m \sigma}_{jk}}{(i\omega)^m },
$$
where $\tau=\tau(x)$ is a real function (eikonal). We insert the
series into the equations under consideration, implement
differentiations and equate the coefficients at the same powers of
the frequency $\omega $ on the left- and right-hand sides of the
so-obtained equalities. In such a way we arrive at the infinite
system of equations
\begin{equation}\label{7.1.8}
{\stackrel m \varepsilon}_{jk}=\frac 12\left({\stackrel m u}_{j;\>
k}+{\stackrel m u}_{k;\> j}+{\stackrel {m+1} u}_{\!\!\!j}\,\tau_{;\>
k}+{\stackrel {m+1} u}_{\!\!\!k}\,\tau_{;\> j}\right)\quad (m=-1,0,\cdots),
\end{equation}
\begin{equation}\label{7.1.10}
{\stackrel m \sigma}_{jk;}{}^k+{\stackrel{m+1}
\sigma}_{\!\!\!jk}\,\tau_{;}{}^k-\rho{\stackrel {m+2}
u_{\!\!\!j}}=0\quad(m=-2,-1,\cdots),
\end{equation}
\begin{equation}  \label{7.1.7}
{\stackrel {-1} \sigma}_ {\!jk}=\lambda {\stackrel {-1} \varepsilon}{}^p_p g_{jk} + 2\mu{\stackrel {-1} \varepsilon}_{\!jk},
\end{equation}
\begin{equation}  \label{7.1.71}
{\stackrel m \sigma}_ {\!jk}=\lambda {\stackrel m \varepsilon}{}^p_p
g_{jk} + 2\mu{\stackrel m \varepsilon}_{jk} +
i(\tau_{;\> j}R_{kp}{\stackrel m u}{}^p+R_{kp}{\stackrel {m-1}
u}{}^p_{;\> j}+c_{jkpq}{\stackrel {m-1} \varepsilon}{}^{pq}) \quad (m
= 0,2,3, ...),
\end{equation}
\begin{equation}  \label{7.1.72}
{\stackrel 1 \sigma}_ {\!jk}=\lambda {\stackrel 1 \varepsilon}{}^p_p
g_{jk} + 2\mu{\stackrel 1 \varepsilon}_{jk} +
i(R_{jk}+\tau_{;\> j}R_{kp}{\stackrel 1 u}{}^p+R_{kp}{\stackrel {0}
u}{}^p_{;\> j}+c_{jkpq}{\stackrel {0} \varepsilon}{}^{pq}) ,
\end{equation}
where it is assumed that
$\stackrel{-1}u=\stackrel{-2}\eps=\stackrel{-2}\sigma=0$. Observe
that $\stackrel {-1} \sigma$ is a symmetric tensor while $\stackrel
m \sigma \ (m\geq 0)$ is not symmetric. Putting $m=-1$ in
\eqref{7.1.8} and $m=-2$ in \eqref{7.1.10}, we have
\begin{equation}\label{7.1.11}
\stackrel{-1}\eps_{jk}=\frac 12\left({\stackrel 0 u}_j\tau_{;\>
k}+{\stackrel 0 u}_k\tau_{;\> j}\right)
\end{equation}
and
\begin{equation}\label{7.1.13}
\stackrel{-1}\sigma_{jk}\tau_{;}{}^k-\rho\stackrel{0}u_j=0.
\end{equation}
Observe that the residual stress does not participate in
\eqref{7.1.7}, \eqref{7.1.11}, and \eqref{7.1.13}. This means that
geometry of rays is the same for the background isotropic medium
(\ref{BG}) and quasi-isotropic medium (\ref{Man1}). As well known (see, for example, \cite[p. 702]{mb}),  \eqref{7.1.7} and \eqref{7.1.11}--\eqref{7.1.13} imply that $\tau$ satisfies the eikonal equation
$$|\nabla\tau|^2=n^2,$$ where
$$n^2=n_p^2=\frac{\rho}{\lambda+2\mu}\quad\text{or}\quad n^2=n_s^2=\frac{\rho}{\mu}.$$

Next, we calculate the (complex) amplitude $A_p$ of a compressional wave by repeating arguments of \cite[Section 7.1.3]{mb}. Recall that the amplitude $A_p$ is defined by
\[
{\stackrel 0 u}=\frac{\lambda+2\mu}{\rho}\langle \stackrel 0 u,\nabla\tau\rangle\nabla\tau=A_p\frac{\nabla\tau}{|\nabla\tau|}.
\]
We fix a solution $\tau$ to the eikonal equation with $n=n_p$ and introduce ray coordinates in a neighborhood of a ray, i.e., curvilinear coordinates $x^1$, $x^2$, $x^3$ such that $x^3=\tau$ and the coordinates surfaces $x^3=x_0^3$ are orthogonal to the coordinates lines $x^1=x_0^1$, $x^2=x_0^2$ that are geodesics of the metric $ds^2=n^2g_{jk}dx^jdx^k$. In such coordinates
\begin{equation}\label{7.1.23}
{\stackrel {-1} \sigma}_ {\!\alpha\beta}=\lambda ng_{\alpha\beta}A_p,\quad
{\stackrel {-1} \sigma}_ {\!\alpha3}=0,\quad
{\stackrel {-1} \sigma}_ {\!33}=(\lambda+2\mu) n^{-1}A_p
\end{equation}
(see \cite[formulas (7.1.23)]{mb}). Here Greek indices assume the values 1,2.  Likewise, we can get that
\begin{equation}\label{7.1.24}
\stackrel 0 \eps_{\alpha\beta}=\frac 12\left(\stackrel 0 u_{\alpha;\> \beta}+\stackrel 0 u_{\beta;\> \alpha}\right),\quad\stackrel 0\eps_{33}=\stackrel 0 u_{3;\> 3}+\stackrel 1 u_3,
\end{equation}
\begin{equation}  \label{7.1.25'}
{\stackrel {0} \sigma}_ {\!33}=\lambda n^{-2}g^{\alpha\beta}{\stackrel {0} \varepsilon}_ {\!\alpha\beta}
+(\lambda+2\mu){\stackrel {0} \varepsilon}_ {\!33}+i(nR_{33}A_p+c_{33pq}{\stackrel {-1} \varepsilon}{}^{pq})
\end{equation}
\begin{equation}\label{7.1.26}
\stackrel{-1}\sigma_{3k;}{}^{k}+n^2\stackrel{0}\sigma_{33}-\rho\stackrel{1}u_3=0.
\end{equation}
Formulas (7.1.27)-(7.1.29) of \cite{mb} remain unchanged, i.e.,
\begin{equation}  \label{7.1.27}
\begin{aligned}
{\stackrel 0 u}_{\alpha\>;\>\beta}&=\frac 1 2 n\frac{\partial g_{\alpha\beta}}{\partial\tau}A_p, \quad
{\stackrel 0 u}_{3\>;\>3}=n^{-1}\frac{\partial A_p}{\partial\tau},\\
{\stackrel{-1}\sigma}_{\!\!3\alpha\>;\>\beta}&=\mu n\frac{\partial g_{\alpha\beta}}{\partial\tau}A_p,\\
{\stackrel{-1}\sigma}_{\!\!3 3\>;\>3}&=(\lambda+2\mu)n^{-1}\frac{\partial A_p}{\partial\tau}+\left[n^{-1}\frac{\partial(\lambda+2\mu)}{\partial\tau }+
(\lambda+2\mu)n^{-2}\frac{\partial n}{\partial\tau}\right]A_p,
\end{aligned}
\end{equation}
\begin{equation}\label{7.1.28}
\stackrel{-1}\sigma_{\!\!3k\>;}{}^k=(\lambda+2\mu)n\frac{\partial A_p}{\partial\tau}+\left(\mu ng^{\alpha\beta}\frac{\partial g_{\alpha\beta}}{\partial\tau}+n\frac{\partial(\lambda+2\mu)}{\partial\tau}+(\lambda+2\mu)\frac{\partial n}{\partial\tau}\right)A_p,
\end{equation}
\begin{equation}\label{7.1.29}
\stackrel{0}\eps_{\alpha\beta}=\frac 12 n\frac{\partial g_{\alpha\beta}}{\partial\tau}A_p,\quad\stackrel{0}\eps_{33}=n^{-1}\frac{\partial A_p}{\partial\tau}+\stackrel{1}u_3.
\end{equation}
Using \eqref{7.1.29} in (\ref{7.1.25'}) implies
$$
{\stackrel 0 \sigma}_{\!33}=(\lambda+2\mu)n^{-1}\frac{\partial A_p}{\partial\tau}+\frac 1 2 \lambda n^{-1}g^{\alpha\beta}\frac{\partial g_{\alpha\beta}}{\partial\tau}A_p
+(\lambda+2\mu){\stackrel 1 u}_3+i(nR_{33}A_p+c_{33pq}{\stackrel {-1} \varepsilon}{}^{pq}).
$$
Inserting this value for ${\stackrel 0 \sigma}_{\!33}$ and value \eqref{7.1.28} for ${\stackrel{-1}\sigma}_{\!\!3 k\>;}{}^k$ into \eqref{7.1.26}, we arrive at the relation
$$
\begin{aligned}
2(\lambda+2\mu)n\frac{\partial A_p}{\partial\tau}+\left[\frac 1 2 (\lambda+2\mu)ng^{\alpha\beta}\frac{\partial g_{\alpha\beta}}{\partial\tau}
+n\frac{\partial(\lambda+2\mu)}{\partial\tau}+(\lambda+2\mu)\frac{\partial n}{\partial\tau}\right] A_p&\\
+[(\lambda+2\mu)n^2-\rho]{\stackrel 1 u}_3+ i(n^3R_{33}A_p+n^2 c_{33pq}{\stackrel {-1} \varepsilon}{}^{ pq })&=0.
\end{aligned}
$$
In view of $n^2=\rho/(\lambda+2\mu)$, the coefficient at ${\stackrel 1 u}_3$ in this formula
is equal to zero. Thus, inserting into the last formula the expressions
$$
{\stackrel {-1} \varepsilon}{}^{\alpha\beta}={\stackrel {-1}\varepsilon}{}^{\alpha 3}=0,\quad
{\stackrel {-1}\varepsilon}{}^{33}=n^3 A_p
$$
that follows from (7.1.11), (5.1.20) and (7.1.22) of \cite{mb},
we arrive at the equation for the amplitude $ A_p $
$$
\frac{\partial A_p}{\partial\tau}+\left(\frac {g^{\alpha\beta}}{4}\frac{\partial g_{\alpha\beta}}{\partial\tau}
+\frac 1 {2({\lambda\!+\!2\mu})}\frac{\partial(\lambda\!+\!2\mu)}{\partial\tau}+\frac 1 {2n} \frac{\partial n}{\partial\tau}
+\frac{in^2}{\lambda\!+\!2\mu}R_{33}+\frac{in^4}{2(\lambda\!+\!2\mu)} c_{3333}\right) A_p=0.
$$
Using notation (5.1.22) and formula (5.1.23) of \cite{mb}, we write this in the form
$$
\frac{\partial}{\partial\tau}\left[\ln\left(A_p J^{1/2}n^{1/2}(\lambda+2\mu)^{1/2}\right)\right]
=-i\frac{n^2}{\lambda+2\mu}R_{33}-i\frac{n^{4}}{2(\lambda+2\mu)} c_{3333},
$$
which, together with $n^2=n_p^2$, gives
$$
\frac{\partial}{\partial\tau}\left[\ln\left(A_p\sqrt{J\rho v_p}\right)\right]=
-\frac{i}{\rho v^4_p}R_{33}-\frac i {2\rho v_p^6}c_{3333},
$$
where $J$ is the geometrical spreading (see \cite[(5.1.22)]{mb}) and $v_p=1/n_p$ is the velocity of compressional waves.
This implies the following analogous of \cite[formula (7.1.32)]{mb}:
\begin{equation}  \label{7.1.32'}
A_p=\frac C{\sqrt{{J\rho v_p}}}\exp\left[
- i\int\limits_{\gamma}\frac 1{\rho v_p^4}R_{jk}\dot\gamma{}^j\dot\gamma{}^k\,d\tau
- i\int\limits_{\gamma}\frac 1{2\rho v_p^6}c_{jklm}\dot\gamma{}^j\dot\gamma{}^k\dot\gamma{}^l\dot\gamma{}^m\,d\tau\right],
\end{equation}
where $\dot\gamma^j=d\gamma^j/d\tau$ and $C$ is a constant for a given ray $\gamma$.

Next, we express the second integrand through $R$. Since $\gamma(\tau)$ is a geodesic of the Riemannian metric $|dx|_p^2=v^{-2}_p|dx|^2_E$, it satisfies
$$
g_{jk}\dot\gamma{}^j\dot\gamma{}^k=|\dot\gamma|^2_E=v^2_p|\dot\gamma|^2_p=v^2_p
$$
under the assumption that $\gamma$ is parameterized by the $\tau$-length, i.e., $|\dot\gamma|^2_p=1$. Taking this into account, we obtain from (\ref{1.4})
\begin{equation}  \label{cRe}
c_{jklm}\dot\gamma{}^j\dot\gamma{}^k\dot\gamma{}^l\dot\gamma{}^m=
v^2_p\Big(2(\nu_3+\nu_4)R_{jk}\dot\gamma{}^j\dot\gamma{}^k+v^2_p(\nu_1+\nu_2)\mbox{tr}_E\,R\Big).
\end{equation}
Here $\mbox{tr}_E\,R$ is the trace of $R$ with respect to the Euclidean metric $(g_{jk})$, i.e., $\mbox{tr}_E\,R=g^{jk}R_{jk}$. Since $\gamma$ is a geodesic of the Riemannian metric
$h_{jk}=v^{-2}_pg_{jk}$, it is more natural to use the Riemannian trace $\mbox{tr}\,R=h^{jk}R_{jk}=v^2_p\,\mbox{tr}_E\,R$. Thus, (\ref{cRe}) takes the form
\begin{equation}  \label{cR}
c_{jklm}\dot\gamma{}^j\dot\gamma{}^k\dot\gamma{}^l\dot\gamma{}^m=
v^2_p\Big(2(\nu_3+\nu_4)R_{jk}\dot\gamma{}^j\dot\gamma{}^k+(\nu_1+\nu_2)\mbox{tr}\,R\Big).
\end{equation}
Substituting this value into (\ref{7.1.32'}), we obtain the final formula for the amplitude of the compressional wave
\begin{equation}  \label{Ap}
A_p=\frac C{\sqrt{{J\rho v_p}}}\exp\left[
- i\int\limits_{\gamma}\frac{1}{\rho v_p^4}\left((1+\nu_3+\nu_4)R_{jk}\dot\gamma{}^j\dot\gamma{}^k
+\frac{1}{2}(\nu_1+\nu_2)\mbox{tr}\,R\right)d\tau\right].
\end{equation}
The physical meaning of the formula is as follows: the residual stress $R$ distorts the phase of the compressional wave.

Finally, we consider propagation of shear waves by repeating
arguments of Sections 7.1.4 and 7.1.5 of \cite{mb}. It can be easily
checked that the residual stress $R$ does not change equations
(7.1.43) of \cite{mb}. Therefore the same formula $A_s=C/\sqrt{J\rho
v_s}$ is valid for the amplitude of the shear wave as for the
background isotropic medium where $v_s=1/n_s$ is the velocity of
shear waves. Moreover, Rytov's law in \cite[formula (7.1.51)]{mb}
does not change either. We reproduce the formula here
\begin{equation}  \label{Rytov}
\left(\frac{D\eta}{d\tau}\right)_j = -i\frac 1 {4\rho
v_s^6}(\delta^q_j-{\dot\gamma}_j\dot\gamma{}^q)c_{qklm}\dot\gamma{}^k\dot\gamma{}^m\eta^l,
\end{equation}
where $\eta$ is the polarization vector defined by ${\stackrel 0
u}_{j}=A_sn_s^{-1}\eta_{j}$. Here $D/d\tau=\dot\gamma{}^j{\nabla}_{\!j}$ with the
covariant derivative taken with respect to the Riemannian metric
$d\tau^2=v^{-2}_s|dx|^2_E=h_{jk}dx^jdx^k$ and indices are raised
with the help of the same metric, i.e.,
$\dot\gamma{}^k=h^{jk}{\dot\gamma}_j$ and $\eta^l=h^{jl}\eta_j$.

\section[The inverse problem for compressional waves]{The inverse
problem for compressional waves}

First of all we emphasize the following feature of our approach.
While considering the forward problem, we impose no boundary condition on the displacement vector $u$ at the boundary of the domain under
consideration. Thus, we treat the problems as if the waves propagate
in an unbounded medium, and use the boundary only as a surface at
which the sources and detectors of oscillations are disposed. In
fact, due to the reflection effects on the boundary, the possibility
of registration of information that is used below as the data for
inverse problems seems to be rather problematic. Here we will not
settle this question but only attract reader's attention to the fact
of its existence.

In contrast to the content of the last paragraph, the following
remark gives the possibility of measuring the data for compressional
waves regardless to any boundary condition. There exists an
alternative version of the geometrical optics method which is based on
the analysis of propagation of the wave front of a non-stationary
elastic wave, see \cite{CJ}. In this version, the integral
participating on (\ref{Ap}) appears as the first order perturbation
for the travel time of a compressional wave. More precisely, for a
fixed geodesic $\gamma$ of the metric $|dx|_p^2=v^{-2}_p|dx|^2_E$
between two boundary points, let $T_{1/\omega}(\gamma)$ be the
propagation time of the compressional wave along $\gamma$ in the
quasi-isotropic medium (\ref{Man1}) and $T_0(\gamma)$ be the
corresponding travel time in the background isotropic medium
(\ref{BG}). Then
$$
T_{1/\omega}(\gamma)-T_0(\gamma)=-\frac{1}{2\omega}D(\gamma)+o(\frac{1}{\omega}),
$$
where
\begin{equation}\label{2.1}
D(\gamma)=\int\limits_{\gamma}\frac{1}{\rho v_p^4}\Big((1+\nu_3+\nu_4)R_{jk}\dot\gamma{}^j\dot\gamma{}^k
+\frac{1}{2}(\nu_1+\nu_2)\mbox{tr}\,R\Big)\,d\tau
\end{equation}
is just the integral participating in \eqref{Ap}. The corresponding result is obtained in \cite{CJ} in the case of a perturbation of the form
$$
\sigma_{jk}=\lambda\varepsilon_{pp}\delta_{jk}+2\mu\varepsilon_{jk}+\frac{1}{\omega}c_{jklm}\varepsilon_{lm}.
$$
By repeating arguments of \cite{CJ}, one easily sees that the result
is true if the last formula is replaced with (\ref{Man4}). Thus,
data (\ref{2.1}) can be obtained by measuring travel times for
compressional waves. Let us consider an elastic wave initiated by a
$\delta$-kind source at the initial point of $\gamma$ which starts
at the time $t=0$. The wave will be a mixture of different body and
surface waves including secondary waves caused by reflections at the
boundary. Nevertheless, $T_{1/\omega}(\gamma)$ is the first arrival
time to the final point of $\gamma$ since compressional waves are
the fastest elastic waves. In the simplest case of constant parameters $\lambda,\mu$, and $\rho$, $T_0(\gamma)$ is equal, up to a constant factor, to the Euclidean distance between the endpoints of the straightline
segment $\gamma$. So, the only problem is the sufficiently precise  measurement of the travel time $T_{1/\omega}(\gamma)$.

Studying the inverse problem, we will first consider the case of constant coefficients $\lambda, \mu,\rho$, $\nu_1,\dots,\nu_4$
since this case is much easier than
the general one and, most probably, this case is of the most
importance for applications.

For the inverse problems studied here, we assume the material
parameters $\lambda, \mu,\rho$, $\nu_1,\dots,\nu_4$ to be given a priori
and only the residual stress $R$ to be unknown.
In practice, some of these parameters are also unknowns to be determined.
Therefore the inverse problem of recovering residual stresses, when some of material parameters $\lambda,\mu,\rho,\nu_1,\dots,\nu_4$ are also unknowns to be recovered, is also worth of investigation.
But this new inverse problem is much harder because it is a {\it nonlinear} problem.


\subsection{The case of constant coefficients}

Let the medium under consideration be contained in a bounded convex
domain $M\subset{\mathbb R}^3$ with smooth boundary $\partial M$ and
let each of the material parameters
$\lambda,\mu,\rho,\nu_1,\dots,\nu_4$ be constant. In this case, the
Riemannian metric $h=v_p^{-2}|dx|^2_E$ is a constant multiple of the
Euclidean metric $g=|dx|^2_E$ and geodesics are intersections of
straight lines with $M$. The equilibrium equation (\ref{1.7}) means
that $R$ is a solenoidal tensor field. We extend $R$ to the whole of
${\mathbb R}^3$ by zero outside $M$. Then the extended tensor field $R$ is
solenoidal on the whole of ${\mathbb R}^3$ in virtue of the boundary condition (\ref{1'}).

We study the inverse problem of recovering the residual stress
tensor field. To this end, assume that we can dispose a source of
compressional waves at every point of the boundary $\partial M $ and
measure the phase of a compressional wave on the same surface
$\partial  M $. In virtue of (\ref{Ap}), our data are integrals (\ref{2.1})
that are known for every line $\gamma$ of ${\mathbb R}^3$. Initially in (\ref{2.1}), $\gamma$ is parameterized by the arc
length in the metric $h$, i.e., $|\dot\gamma|_E=v_p=\mbox{const}$ and the trace is understood with respect to the metric $h$, i.e., $\mbox{tr}\,R=v^2_p\mbox{tr}_ER$. After a simple rescaling, we obtain the same formula (\ref{2.1}), where now $|\dot\gamma|_E=1$ and $\mbox{tr}\,R$ is replaced by the Euclidean trace $\mbox{tr}_ER$. From now on in this subsection, we use the Euclidean metric only and write $\mbox{tr}\,R$ instead of $\mbox{tr}_ER$.

To get a well defined inverse problem, we have to impose some
restrictions on the material parameters. Indeed, if for example
$\nu_1+\nu_2=-2(1+\nu_3+\nu_4)/3$, then the integrand on (\ref{2.1})
is identically equal to zero for $R=g$.
Therefore, in this section, we assume that
\begin{equation}\label{nzero1}
3(\nu_1+\nu_2)+2(1+\nu_3+\nu_4)\neq 0.
\end{equation}
 The first term of the
integrand on (\ref{2.1}) is considered as the leading term. Therefore, in this section, we also assume that
\begin{equation}\label{nzero2}
1+\nu_3+\nu_4\neq 0.
\end{equation}
Introducing the notations
\begin{equation}  \label{2.2}
f=\frac{1+\nu_3+\nu_4}{\rho v^4_p}R,\quad a=\frac{\nu_1+\nu_2}{2(1+\nu_3+\nu_4)}
\end{equation}
($f$ is zero outside $M$), we write data (\ref{2.1}) as
\begin{equation}  \label{2.3}
(I(f+a(\mbox{tr}\,f)g))(\gamma)=\int\limits_{\gamma}(f_{jk}+a(\mbox{tr}\,f)g_{jk})\dot\gamma{}^j\dot\gamma{}^k\,d\tau,
\end{equation}
where $I$ is the (longitudinal) ray transform on ${\mathbb R}^3$ which is defined in Section 2.1 of \cite{mb}.

Let us remind the theorem on decomposition of a tensor field into solenoidal and potential parts (Theorem 2.6.3 of \cite{mb}): every symmetric tensor field $u=(u_{jk})\in L^2$ on ${\mathbb R}^3$ can be uniquely represented as
$$
u=\tilde u+dv,\quad \delta\tilde u=0,
$$
where $v$ is continuous outside $M$ and satisfies $v(x)\rightarrow 0$ as $|x|\rightarrow\infty$. Here the operators $d$ (inner derivative) and $\delta$ (divergence) are defined in local coordinates by formulas $(dv)_{jk}=({\nabla}_{\!j}v_k+{\nabla}_{\!k}v_j)/2$ and
$(\delta u)_j=g^{kl}{\nabla}_{\!k}u_{jl}$ respectively, $\nabla$ being the covariant derivative with respect to the Euclidean metric $g$. The summands $\tilde u$ and $dv$ of the decomposition are called the solenoidal and potential parts of the tensor field $u$ respectively.

 We now investigate the question of uniqueness of a solution to the inverse problem. Let a solenoidal field $f$ satisfy $I(f+a(\mbox{tr}\,f)g)=0$, $a=\mbox{const}$. By Theorem 2.15.1 of \cite{mb}, this means that the solenoidal part of $f+a(\mbox{tr}\,f)g$ is equal to zero. Since the solenoidal part of $f$ coincides with $f$, we obtain the equation
\begin{equation}  \label{2.9'}
f+aS((\mbox{tr}\,f)g)=0,
\end{equation}
where $S(u)$ denotes the solenoidal part of a tensor field $u$. Recall that the Fourier transform interweaves the operators $S$ and $T$, where $Tu$ stands for the tangential part of $u$, see Section 2.6 of \cite{mb} for details. Applying the Fourier transform to the last equation, we obtain
\begin{equation}  \label{2.10}
\hat f+a(\mbox{tr}\,\hat f)\varepsilon=0,
\end{equation}
where the tensor field $\varepsilon=T(g)$ is expressed in Cartesian coordinates by $\varepsilon_{jk}(y)=\delta_{jk}-y_jy_k/|y|^2$, $y$ being the variable in the Fourier space.

Applying the operator $\mbox{tr}$ to equation (\ref{2.10}), we obtain $(1+2a)\mbox{tr}\,\hat f=0$. Thus, the inequality $1+2a\neq 0$ is the necessary and sufficient condition for the uniqueness of a solution to the inverse problem. Recalling (\ref{2.2}), we write the condition as
\begin{equation}  \label{2.11}
\nu_1+\nu_2+\nu_3+\nu_4\neq -1.
\end{equation}
If (\ref{2.11}) holds, equation $(1+2a)\mbox{tr}\,\hat f=0$ gives $\mbox{tr}\,\hat f=0$. Then (\ref{2.10}) implies $\hat f=0$ and therefore $f=0$.

Under hypothesis (\ref{2.11}), an explicit inversion formula for recovering a solenoidal tensor field $f$ from the data  $I(f+a(\mbox{tr}\,f)g)$ can be easily derived from the corresponding inversion formula for $I$, see Theorem 2.12.2 of \cite{mb}. The corresponding stability estimate can be also obtained on the base of the Plancherel formula for the ray transform, see Section 2.15 of \cite{mb}.
Moreover, to recover a solenoidal $f$, we do not need to measure ray integrals $(I(f+a(\mbox{tr}\,f)g))(\gamma)$ for all lines $\gamma$ of ${\mathbb R}^3$. Repeating arguments of \cite{Sh2}, we see that three families of lines are sufficient for an effective reconstruction algorithm, each family consists of all lines parallel to a coordinate plane.

If (\ref{2.11}) does not hold, i.e., if $\nu_1+\nu_2+\nu_3+\nu_4= -1$, then the space of solenoidal tensor fields $f$ satisfying $I(f+a(\mbox{tr}\,f)g)=0$ can be explicitly described. Indeed, in this case $\mbox{tr}\,f$ can be an arbitrary function and equation (\ref{2.9'}) gives $f=S(\alpha g)$ with an arbitrary scalar function $\alpha$.

\subsection{The case of variable coefficients}

Let again $M\subset{\mathbb R}^3$ be a closed bounded domain with smooth boundary $\partial M$. Now the
material parameters $\lambda,\mu,\rho,\nu_1,\dots,\nu_4$ are assumed to be
known smooth functions of a point $x\in M$. Let
$v_p=\sqrt{(\lambda+2\mu)/\rho}$ be the velocity of compressional
waves. By $g$ we denote the Euclidean metric and by $h=v^{-2}_pg$,
the Riemannian metric corresponding to compressional waves. Assume
$(M,h)$ to be a convex non-trapping manifold (CNTM) in the sense of
definition given in \cite{Sh1}. The same definition is presented in
Section 4.1 of \cite{mb} but the term ``compact dissipative
Riemannian manifold'' is used instead of CNTM.

We consider the inverse problem of recovering the residual stress
tensor field. Our data are integrals \eqref{2.1}
that are known for every geodesic $\gamma$ of the metric $h$ with
endpoints in $\partial M$. The geodesic is parameterized by the arc
length in the metric $h$, i.e., $|\dot\gamma|=1$. By the same
arguments as above, we assume inequalities (\ref{nzero1}) and
(\ref{nzero2}) to be valid everywhere in $M$. On using the same
notations (\ref{2.2}), we write data (\ref{2.1}) as
\begin{equation}  \label{2.3'}
I(f+a(\mbox{tr}\,f)h)=\int\limits_{\gamma}(f_{jk}+a(\mbox{tr}\,f)h_{jk})\dot\gamma{}^j\dot\gamma{}^k\,d\tau,
\end{equation}
where $I$ is the (longitudinal) ray transform on the CNTM $(M,h)$ which is defined in Section 4.2 of \cite{mb}.

For a compact Riemannian manifold $(M,h)$, the theorem on decomposition of a tensor field into solenoidal and potential parts (Theorem 3.3.2 of \cite{mb}) is valid in the following form: every symmetric tensor field $u=(u_{jk})$ can be uniquely represented as
$$
u=\tilde u+dv,\quad \delta\tilde u=0,\quad v|_{\partial M}=0
$$
where the operators $d$ (inner derivative) and $\delta$ (divergence) are defined in local coordinates by the same formulas $(dv)_{jk}=({\nabla}_{\!j}v_k+{\nabla}_{\!k}v_j)/2$ and
$(\delta u)_j=h^{kl}{\nabla}_{\!k}u_{jl}$ respectively, $\nabla$ being the covariant derivative with respect to the metric $h$.

According to Theorem 4.3.3 of \cite{mb}, the solenoidal part of the tensor field $f+a(\mbox{tr}\,f)h$ can be uniquely recovered from data (\ref{2.3'}) under the assumption
\begin{equation}  \label{k+}
k^+(M,h)<1/3,
\end{equation}
where $k^+(M,h)$ is some curvature characteristic of the CNTM $(M,h)$. For the metric $h=v^{-2}_pg$, condition (\ref{k+}) holds if the function $v_p$ is sufficiently $C^2$-close to a constant, the degree of the closeness depends on the size of the domain $M$.
 For such a manifold, the null-space of $I$ consists of potential fields that can be represented in the form $dv$ with a covector field $v$ satisfying the boundary condition $v|_{\partial M}=0$.

For variable coefficients, the main difficulty relates to the following circumstance: the equilibrium condition (\ref{1.7}) does not mean anymore that $f$ is a solenoidal tensor field with respect to the metric $h$. Indeed,
 (\ref{1.7}) can be rewritten in terms of $f$ as
\begin{equation}  \label{2.4}
\delta_E(bf)=0\quad\mbox{with}\quad b=\frac{\rho v^4_p}{1+\nu_3+\nu_4},
\end{equation}
where $\delta_E$ is the divergence with respect to the Euclidean
metric $g$. We are going to prove, at least under some restrictions
on the material parameters $\lambda,\mu,\rho,\nu_1,\dots,\nu_4$,
that a tensor field $f$ is uniquely determined by data (\ref{2.3'})
if it satisfies (\ref{2.4}).

First of all we will rewrite equation (\ref{2.4}) in terms of the divergence $\delta f$ with respect to the metric $h$. Denote $c=v^2_p$, then $g=ch$. On using standard formulas of tensor analysis, one easily calculates
$$
(\delta_E(bf))_j=bc^{-1}(\delta f)_j+c^{-1}\Big(f_{jk}{\nabla}^kb-\frac{1}{2}bc^{-1}((\mbox{tr}\,f){\nabla}_{\!j}c-f_{jk}{\nabla}^kc)\Big),
$$
where the covariant derivative and trace are understood with respect to the metric $h$. Therefore (\ref{2.4}) is equivalent to the equation
$$
(\delta f)_j+b^{-1}f_{jk}{\nabla}^kb+\frac{1}{2}c^{-1}f_{jk}{\nabla}^kc-\frac{1}{2}c^{-1}(\mbox{tr}\,f){\nabla}_{\!j}c=0.
$$
Denoting
\begin{equation}  \label{2.5}
\alpha=b^{-1}{\nabla}b+\frac{1}{2}c^{-1}{\nabla}c,\quad \beta=-\frac{1}{2}c^{-1}{\nabla}c,
\end{equation}
we write the equation in the coordinate free form
\begin{equation}  \label{2.6}
\delta f+f\alpha+(\mbox{tr}\,f)\beta=0.
\end{equation}

We now investigate the uniqueness question to the inverse problem.
Let a tensor field $f$ be such that $I(f+a(\mbox{tr}\,f)h)=0$. We
assume that the Riemannian manifold $(M,h)$ is a CNTM and satisfies
the curvature condition (\ref{k+}). Then, by Theorem 4.3.3 of \cite{mb},
$f+a(\mbox{tr}\,f)h$ must be a potential field, i.e., there exists a
covector field $v$ on $M$ satisfying the boundary condition
$v|_{\partial M}=0$ such that
$$
f+a(\mbox{tr}\,f)h=dv.
$$
Taking the trace of both parts, we obtain
$$
(1+3a)(\mbox{tr}\,f)=\delta v.
$$
By \eqref{nzero1}, $1+3a$ does not vanish on $M$ and
we can write
$$
\mbox{tr}\,f=(1+3a)^{-1}\delta v,\quad f=dv-a(1+3a)^{-1}(\delta v)h.
$$
Substituting these values into (\ref{2.6}), we arrive at the
boundary value problem on the covector field $v$
\begin{equation}  \label{2.7}
\left\{
\begin{aligned}
&(\frac{a}{1+3a}d\delta-\delta d)v-(dv)\alpha+\frac{1}{1+3a}(\delta v)\Big(a\alpha-\beta+(1+3a){\nabla}(\frac{a}{1+3a})\Big)=0,\\
&v|_{\partial M}=0.
\end{aligned}
\right.
\end{equation}
We have thus proved

\begin{prop}
Let a three-dimensional CNTM $(M,h)$ satisfy the curvature condition
(\ref{k+}) and let $a,\alpha$, and $\beta$ be defined by (\ref{2.2})
and (\ref{2.5}). Every symmetric tensor field $f$ satisfying
(\ref{2.6}) can be uniquely recovered from data (\ref{2.3'}) if and
only if the boundary value problem (\ref{2.7}) has no nontrivial
solution.
\end{prop}

Equation (\ref{2.7}) is rather complicated in the case of general coefficients $a,\alpha$, and $\beta$. We are going to investigate the boundary value problem in the case when the background medium is sufficiently close to a homogeneous one, i.e., when $a$ is close to a constant and $\alpha$ and $\beta$ are small. Even in this case, we need to impose some restrictions on $a$.

Let us first find a condition that guarantees ellipticity of the
boundary value problem.
\begin{prop}\label{elliptic}
The operator
$\frac{a}{1+3a}d\delta-\delta d$ is elliptic if and only if the inequality
\begin{equation}  \label{2.9}
|1+a|<|1+3a|
\end{equation}
holds on the whole of $M$.
\end{prop}
\begin{proof}
Recall that the principal symbols of the operators $d$ and $\delta$ are
$\sqrt{-1}i_\xi$ and $\sqrt{-1}j_\xi$ respectively, where $i_\xi$ is
the symmetric multiplication by the covector $\xi$ and $j_\xi$ is
the contraction with $\xi$, see \cite[Section~3.3]{mb} for details.
Therefore the principal symbol of our operator is
$$
\sigma\Big(\frac{a}{1+3a}d\delta-\delta d\Big)=\Big(j_\xi i_\xi-\frac{a}{1+3a}i_\xi j_\xi\Big).
$$
By \cite[Lemma~3.3.3]{mb}, $j_\xi i_\xi=\frac{1}{2}(|\xi|^2E+i_\xi
j_\xi)$ on covectors, where $E$ is the identity operator. Therefore
\begin{equation}  \label{2.8}
\sigma\Big(\frac{a}{1+3a}d\delta-\delta d\Big)=\frac{1}{2}(|\xi|^2E+\kappa i_\xi j_\xi)\quad\mbox{with}\quad \kappa=\frac{1+a}{1+3a}.
\end{equation}
On the other hand, for a covector $v$,
$$
\langle(|\xi|^2E+\kappa i_\xi j_\xi)v,v\rangle=|\xi|^2|v|^2+\kappa\langle\xi,v\rangle^2.
$$
Hence, \eqref{2.9} guarantees the positiveness of
$|\xi|^2E+\kappa i_\xi j_\xi$ for all $\xi\neq 0$.
\end{proof}

Under suitable assumptions on coefficients, the triviality of a solution to the
boundary value problem \eqref{2.7} is guaranteed by the following

\begin{thm}\label{ut}
Given a three-dimensional CNTM $(M,h)$, let $D$ be its diameter,
i.e., the length of the longest geodesic. Assume the coefficients
$a,\alpha$, and $\beta$ of equation \eqref{2.7} to satisfy
\begin{equation}\label{mi}
3a_0+\frac{1}{2}\alpha_0+\frac{3}{2}\beta_0+\frac{1}{4}(\alpha_0^3+\beta_0^3)D^2<1,
\end{equation}
where
$$
a_0=\sup\left|\frac{a}{1+3a}\right|,\quad\alpha_0=\Big(\sup|\alpha|\Big)^{1/2},\quad\beta_0=\Big(\sup\left|\frac{a\alpha-\beta}{1+3a}\right|\Big)^{1/2}.
$$
Then the boundary value problem \eqref{2.7} has only trivial solution.
\end{thm}

The values of $\alpha_0$ and $\beta_0$ can be made arbitrary small by assuming the background medium to be sufficiently close to a homogeneous one. Therefore the main part of hypothesis (\ref{mi}) is $a_0<1/3$ that is equivalent to
\begin{equation}\label{ui}
a>-1/6
\end{equation}
It is interesting to compare \eqref{ui} with the ellipticity
condition \eqref{2.9}. If $-1/6<a<0$, then the boundary value
problem is not elliptic but has only trivial solution for
sufficiently small $\alpha$ and $\beta$. If both conditions
\eqref{2.9} and \eqref{mi} are satisfied, we can use the standard
stability estimate for the elliptic boundary value problem with the
trivial kernel to obtain some stability estimate in the inverse
problem of recovering a tensor field $f$ from data \eqref{2.3'}.
This stability estimate will be similar to that of \cite[Theorem
4.3.4]{mb}. We do not present it here.

To prove Theorem \ref{ut} we need the following

\begin{lem}\label{L3.1}
For a covector field $v$ on a three-dimensional CNTM $(M,h)$ satisfying $v|_{\partial M}=0$, the estimate
$$
\|v\|^2_{L^2}\leq\frac{D^2}{10}(2\|dv\|^2_{L^2}+\|\delta v\|^2_{L^2})
$$
holds where $D$ is the diameter of $(M,h)$.
\end{lem}

\begin{proof}
Let $\Omega M=\{(x,\xi)\mid x\in M,\xi\in T_xM,|\xi|^2=h_{ij}\xi^i\xi^j=1\}$ be the unit sphere bundle. Introduce the $L^2$-norm
$$
\|u\|^2_{L^2(\Omega M)}=\frac{1}{4\pi}\int\limits_{\Omega
M}|u(x,\xi)|^2\,d\omega_x(\xi)dV(x),
$$
where $d\omega_x$ is the volume form on the sphere $\Omega_x=\Omega
M\cap T_xM$ and $dV$ is the Riemannian volume form on $M$.

Given a covector field $v$ on $M$, define two functions on $\Omega M$
$$
\varphi(x,\xi)=v_i(x)\xi^i,\quad \psi(x,\xi)=(dv(x))_{ij}\xi^i\xi^j.
$$
The functions are related by the equation $H\varphi=\psi$ where $H$ is the differentiation with respect to the geodesic flow, see \cite[Section 4.4]{mb} for details. If $v|_{\partial M}=0$, then
$\varphi$ vanishes on the boundary of $\Omega M$. Applying the Poincar\'{e} inequality \cite[Lemma 4.5.1]{mb} with the weight $\lambda\equiv 1$, we obtain
\begin{equation}  \label{3.21}
\|\varphi\|^2_{L^2(\Omega M)}\leq\frac{D^2}{2}\|H\varphi\|^2_{L^2(\Omega M)}=\frac{D^2}{2}\|\psi\|^2_{L^2(\Omega M)}.
\end{equation}

One can easily see that (compare with \cite[Lemma 7.4.2]{mb}) in the three-dimensional case
$$
 \|\varphi\|^2_{L^2(\Omega M)}=\frac{1}{3}\|v\|^2_{L^2(M)},
 $$
 $$
  \|\psi\|^2_{L^2(\Omega M)}=\frac{1}{15}(2\|dv\|^2_{L^2(M)}+\|jdv\|^2_{L^2(M)})=\frac{1}{15}(2\|dv\|^2_{L^2(M)}+\|\delta v\|^2_{L^2(M)}).
 $$
 Inserting these values into (\ref{3.21}), we get the statement of the lemma.
\end{proof}

{\it Proof of Theorem} \ref{ut}.
Let $v$ be a solution to the boundary value problem \eqref{2.7}. First of all we rewrite equation \eqref{2.7} in the form
$$
d\Big(\frac{a}{1+3a}\delta v\Big)-\delta dv-(dv)\alpha+(\delta v)\frac{a\alpha-\beta}{1+3a}=0.
$$
Take the $L^2$-product of the equation with $v$, use the boundary condition $v|_{\partial M}=0$ and the fact that $d$ and $-\delta$ are dual operators (see \cite[Theorem 3.3.1]{mb})
\begin{equation}  \label{3.22}
\|dv\|^2=(\frac{a}{1+3a}\delta v,\delta v)+((dv)\alpha,v)-(\frac{a\alpha-\beta}{1+3a}\delta v,v).
\end{equation}

We estimate each summand on the right-hand side of \eqref{3.22}. Obviously
\begin{equation}  \label{3.23}
\left|(\frac{a}{1+3a}\delta v,\delta v)\right|\leq a_0\|\delta v\|^2.
\end{equation}
The second summand on the right-hand side of \eqref{3.22} is
estimated as follows:
$$
|((dv)\alpha,v)|\leq(\alpha_0^{1/2}\|dv\|)(\alpha_0^{3/2}\|v\|)\leq
\frac{1}{2}(\alpha_0\|d v\|^2+\alpha_0^3\|v\|^2).
$$
This implies with the help of Lemma \ref{L3.1}
\begin{equation}  \label{3.24}
|((dv)\alpha,v)|\leq\frac{5\alpha_0+\alpha_0^3D^2}{10}\|dv\|^2+\frac{\alpha_0^3D^2}{20}\|\delta v\|^2.
\end{equation}
The last term on the right-hand side of \eqref{3.22} is estimated in the same way:
\begin{equation}  \label{3.25}
\left|(\frac{a\alpha-\beta}{1+3a}\delta v,\delta v)\right|\leq\frac{\beta_0^3D^2}{10}\|dv\|^2+\frac{10\beta_0+\beta_0^3D^2}{20}\|\delta v\|^2.
\end{equation}

With the help of \eqref{3.23}--\eqref{3.25}, \eqref{3.22} gives
\begin{equation}  \label{3.26}
\|dv\|^2\leq\frac{1}{10}(5\alpha_0+\alpha_0^3D^2+\beta_0^3D^2)\|dv\|^2+\Big(a_0+\frac{1}{20}(10\beta_0+\alpha_0^3D^2+\beta_0^3D^2)\Big)\|\delta v\|^2.
\end{equation}

The inequality $\|\delta v\|^2\leq 3\|dv\|^2$ is obvious since
$\delta v$ is the trace of the $3\times3$-matrix $((dv)_{ij})$. With
the help of the latter inequality, \eqref{3.26} gives
$$
\Big[1-3a_0-\frac{1}{2}\alpha_0-\frac{3}{2}\beta_0-\frac{1}{4}(\alpha_0^3+\beta_0^3)D^2\Big]\|dv\|^2\leq 0.
$$
The coefficient in the brackets is positive under hypothesis \eqref{mi} and the inequality implies $dv\equiv0$. With the help of the boundary condition $v|_{\partial M}=0$, this implies
$v\equiv0$. \hfill $\Box$

\section[The inverse problem for shear waves]{The inverse problem for shear waves}

Let $v_s=\sqrt{\mu/\rho}$ be the velocity of shear waves. By $g$ we
denote the Euclidean metric and by $h=v^{-2}_sg$, the Riemannian
metric corresponding to shear waves. Assume $(M,h)$ to be a CNTM. We
assume that, for every geodesic $\gamma:[0,l]\rightarrow M,\
|\dot\gamma|^2=h_{jk}\dot\gamma{}^j\dot\gamma{}^k=1$ with endpoints
in $\partial M$, we can activate the shear wave with an arbitrary
initial polarization $\eta(0)$ and can measure the final
polarization $\eta(l)$. Mathematically, this means that the
fundamental matrix $U(\gamma)$ of system (\ref{Rytov}) is known
for every geodesic $\gamma$ with endpoints in $\partial M$ such that
$\eta(l)=U(\gamma)\eta(0)$. We are going to recover the residual
stress $R$ from the data $U(\gamma)$.

Introduce the tensor field $ f $ by
\begin{equation}  \label{3.1}
f_{jklm}=-i\frac 1 {4\rho v_s^6}(c_{jlkm}+c_{jmkl}).
\end{equation}
It possesses the symmetries
$$
f_{jklm}=f_{kjlm}=f_{jkml}=f_{lmjk}
$$
as follows from (\ref{1.6}). Express $f$ through $R$ by substituting (\ref{1.4}) into (\ref{3.1})
\begin{equation}  \label{3.2}
\begin{aligned}
f_{jklm}=-i\frac 1 {4\rho v_s^4}\Big(&
\nu_1(\mbox{tr}\,R)(h_{jl}h_{km}+h_{jm}h_{kl})
+\frac{\nu_2}{2}(\mbox{tr}\,R)(2h_{jk}h_{lm}+h_{jl}h_{km}+h_{jm}h_{kl})\\
+&\nu_3(R_{jl}h_{km}+R_{jm}h_{kl}+R_{kl}h_{jm})\\
+&\frac{\nu_4}{2}(2R_{jk}h_{lm}+R_{jl}h_{km}+R_{jm}h_{kl}+R_{kl}h_{jm}+R_{km}h_{jl}+2R_{lm}h_{jk})\Big).
\end{aligned}
\end{equation}
Here $\mbox{tr}\,R=h^{jk}R_{jk}$ is the Riemannian trace.

Rytov's law (\ref{Rytov}) is written in terms of $f$ as
\begin{equation}  \label{3.4}
\left(\frac{D\eta}{d\tau}\right)_j = (\delta^p_j-{\dot\gamma}_j\dot\gamma{}^p)f_{qrjk}\dot\gamma{}^q\dot\gamma{}^r\eta^k.
\end{equation}
To write Rytov's law in a coordinate free form, we introduce the following notations. Define the linear operator $f_{\dot\gamma}$ by
$(f_{\dot\gamma}\eta)_j=f_{qrjk}\dot\gamma{}^q\dot\gamma{}^r\eta^k$. For a nonzero vector $\xi$, let $P_\xi$ be the orthogonal projection onto
$\xi^\bot=\{\eta\mid\langle\xi,\eta\rangle=\xi^j\eta_j=0\}$. In coordinates $(P_\xi\eta)_j=(\delta^p_j-\frac{1}{|\xi|^2}\xi_j\xi^p)\eta_p$. Then the coordinate free form of (\ref{3.4}) is
$$
\frac{D\eta}{d\tau}=P_{\dot\gamma}(f_{\dot\gamma}\eta).
$$
This equation can be also written in the form
\begin{equation}  \label{3.5}
\frac{D\eta}{d\tau}=(P_{\dot\gamma}f_{\dot\gamma})\eta,\quad \eta(0)=\eta_0,
\end{equation}
where, for a symmetric tensor $u=(u_{jk})$,
$$
(P_\xi u)_{jk}=(\delta^p_j-\frac{1}{|\xi|^2}\xi_j\xi^p)(\delta^q_k-\frac{1}{|\xi|^2}\xi_k\xi^q)u_{pq},
$$
compare with \cite[formula (5.2.1)]{mb}. $P_\xi$ is the orthogonal
projection onto the subspace $\{u\mid u_{jk}\xi^k=0\}$. Let
$\dot\gamma{}^\bot(\tau)$ be the (complex) two-dimensional subspace
of (the complexification of) the tangent space $T_{\gamma(\tau)}M$ consisting of vectors orthogonal to $\dot\gamma(\tau)$,
and let
$I^{0,\tau}_\gamma:\dot\gamma{}^\bot(0)\rightarrow\dot\gamma{}^\bot(\tau)$
be the parallel transport along $\gamma$ with respect to the metric
$h$. The solution to the initial value problem (\ref{3.5}) can be
written as
\begin{equation}  \label{3.6}
\eta(l)=I^{0,l}_\gamma(U(\gamma)\eta_0)\quad\mbox{for}\quad\eta_0\in\dot\gamma{}^\bot(0)
\end{equation}
with some linear operator (the fundamental matrix) $U(\gamma):\dot\gamma{}^\bot(0)\rightarrow\dot\gamma{}^\bot(0)$. Observe that $U(\gamma)$ is a unitary operator since the matrix of system (\ref{3.5}) is skew-Hermitian as is seen from (\ref{3.2}). We consider the problem of recovering the tensor field $R$ from the data $U(\gamma)$ known for all geodesics with endpoints in $\partial M$.

The problem is strongly nonlinear since the data $U(\gamma)$ depends on $R$ in a nonlinear manner. Let us linearize the problem. To this end we represent the fundamental matrix by the Neumann series
$$
U(\gamma)=E+\int\limits_0^l I_\gamma^{\tau,0}((P_{\dot\gamma}f_{\dot\gamma})(\tau))\,d\tau+
\int\limits_0^l I_\gamma^{\tau,0}((P_{\dot\gamma}f_{\dot\gamma})(\tau))\,d\tau\int\limits_0^\tau I_\gamma^{t,\tau}((P_{\dot\gamma}f_{\dot\gamma})(t))\,dt+\dots,
$$
where $E$ is the identity matrix, and delete the terms that are nonlinear in $f$. In other words, we use Born's approximation
$$
U(\gamma)-E=\int\limits_0^l I_\gamma^{\tau,0}((P_{\dot\gamma}f_{\dot\gamma})(\tau))\,d\tau
$$
as the data for the linearized inverse problem. Since the integrand is a symmetric operator, the data are equivalent to the quadratic form
\begin{equation}  \label{3.7}
\langle(U(\gamma)-E)\eta(0),\eta(0)\rangle=\int\limits_0^l \langle ((P_{\dot\gamma}f_{\dot\gamma}(\tau))\eta(\tau),\eta(\tau)\rangle\, d\tau
\end{equation}
on the two-dimensional vector space of vector fields $\eta(\tau)$ that are orthogonal to $\dot\gamma$ and parallel along $\gamma$ in the sense of the metric $h$, i.e., satisfy
$\frac{D\eta}{d\tau}=0$. We denote this space by $\gamma^\bot$. Since
$$
\langle (P_{\dot\gamma}f_{\dot\gamma})\eta,\eta\rangle=\langle
P_{\dot\gamma}(f_{\dot\gamma}\eta),\eta\rangle=\langle
f_{\dot\gamma}\eta,P_{\dot\gamma}\eta\rangle= \langle
f_{\dot\gamma}\eta,\eta\rangle=
f_{jklm}\dot\gamma{}^j\dot\gamma{}^k\eta^l\eta^m,
$$
(\ref{3.7}) can be written as
\begin{equation}  \label{3.8}
\langle(U(\gamma)-E)\eta(0),\eta(0)\rangle=(Lf)(\gamma,\eta):=\int\limits_0^l f_{jklm}(\gamma(\tau))\dot\gamma{}^j(\tau)\dot\gamma{}^k(\tau)\eta^l(\tau)\eta^m(\tau)\, d\tau.
\end{equation}
The operator $L$ defined by this formula is called the mixed ray
transform (compare with \cite[formula (7.1.55)]{mb}).

We express the integrand of (\ref{3.8}) through $R$ by substituting value (\ref{3.2}) for $f$. On using the relations $|\dot\gamma|=1$ and $\langle\dot\gamma,\eta\rangle=0$, we obtain
\begin{equation}  \label{3.9}
(Lf)(\gamma,\eta)=-i\int\limits_\gamma\frac{1}{4\rho v_s^4}\Big(
\nu_4R_{jk}\eta^j\eta^k+\nu_4R_{jk}\dot\gamma{}^j\dot\gamma{}^k|\eta|^2+\nu_2(\mbox{tr}\,R)|\eta|^2\Big)\, d\tau.
\end{equation}
The first term of the integrand on (\ref{3.9}) is considered as the
leading term. Therefore we assume in this section that the function
$\nu_4$ does not vanish in $M$. Note that we do not use the
equilibrium condition (\ref{1.7}) in this section. Introducing the
notations
\begin{equation}  \label{3.10}
F=-\frac{i\nu_4}{4\rho v_s^4}R,\quad a=\frac{\nu_2}{\nu_4},
\end{equation}
we write (\ref{3.9}) as
\begin{equation}  \label{3.11}
(Lf)(\gamma,\eta)=\int\limits_\gamma F_{jk}\eta^j\eta^k\,
d\tau+|\eta|^2\int\limits_\gamma
F_{jk}\dot\gamma{}^j\dot\gamma{}^k\, d\tau+
|\eta|^2\int\limits_\gamma a(\mbox{tr}\,F)\, d\tau.
\end{equation}
The first term on the right-hand side of (\ref{3.11}) is the
transverse ray transform $(JF)(\gamma,\eta)$, compare with
\cite[formula (5.1.72)]{mb}. The second term coincides, up to the
factor $|\eta|^2$, with the longitudinal ray transform
$(IF)(\gamma)$, while the last term coincides with
$I(a(\mbox{tr}\,F)h)(\gamma)$. Therefore (\ref{3.11}) can be written
as
\begin{equation}  \label{3.12}
(JF)(\gamma,\eta)+|\eta|^2(IF)(\gamma)+|\eta|^2I(a(\mbox{tr}\,F)h)(\gamma)=(Lf)(\gamma,\eta)\quad
(\eta\in\gamma^\bot).
\end{equation}
We consider (\ref{3.12}) as an equation in an unknown symmetric tensor field $F=(F_{jk})$ while the right-hand side $(Lf)(\gamma,\eta)$ is given.

Equation (\ref{3.12}) can be simplified. Indeed, let $(\eta_1,\eta_2)$ be an orthonormal basis of $\gamma^\bot$. Then
\begin{equation}  \label{3.13}
(JF)(\gamma,\eta_1)+(JF)(\gamma,\eta_2)+2(IF)(\gamma)+2I(a(\mbox{tr}\,F)h)(\gamma)=(Lf)(\gamma,\eta_1)+(Lf)(\gamma,\eta_2).
\end{equation}
Since
$$
(JF)(\gamma,\eta_1)+(JF)(\gamma,\eta_2)=\int\limits_\gamma(\mbox{tr}\,P_{\dot\gamma}F)\, d\tau,
$$
(\ref{3.13}) can be written as
$$
\frac{|\eta|^2}{2}\int\limits_\gamma(\mbox{tr}\,P_{\dot\gamma}F)\, d\tau+|\eta|^2(IF)(\gamma)+|\eta|^2I(a(\mbox{tr}\,F)h)(\gamma)=
\frac{|\eta|^2}{2}\Big((Lf)(\gamma,\eta_1)+(Lf)(\gamma,\eta_2)\Big).
$$
Subtracting this equality from (\ref{3.12}), we obtain
\begin{equation}  \label{3.14}
(JF)(\gamma,\eta)-\frac{|\eta|^2}{2}\int\limits_\gamma(\mbox{tr}\,P_{\dot\gamma}F)\, d\tau=(Lf)(\gamma,\eta)-
\frac{|\eta|^2}{2}\Big((Lf)(\gamma,\eta_1)+(Lf)(\gamma,\eta_2)\Big).
\end{equation}
Let $Q_\xi$ be the orthogonal projection of symmetric tensors onto
the subspace $\{u=(u_{jk})\mid u_{jk}\xi^k=0,\ \mbox{tr}\,u=0\}$
(compare with Section 6.2 of \cite{mb}). The left-hand side of
(\ref{3.14}) can be transformed as follows:
$$
\begin{aligned}
(JF)(\gamma,\eta)-\frac{|\eta|^2}{2}\int\limits_\gamma(\mbox{tr}\,P_{\dot\gamma}F)\, d\tau&=
\int\limits_\gamma\Big(\langle(P_{\dot\gamma}F)\eta,\eta\rangle-\frac{|\eta|^2}{2}(\mbox{tr}\,P_{\dot\gamma}F)\Big)\, d\tau\\&=
\int\limits_\gamma\langle(Q_{\dot\gamma}F)\eta,\eta\rangle\, d\tau=(KF)(\gamma,\eta),
\end{aligned}
$$
where $K$ is the truncated transverse ray transform, see the definition in Section 6.2 of \cite{mb}. Since $(Lf)(\gamma,\eta)$ is known for every $\eta\in\gamma^\bot$, the right-hand side of
(\ref{3.14}) is known too. Moreover, it is independent of the choice of an orthonormal basis $(\eta_1,\eta_2)$ as follows from (\ref{3.14}). Denoting the right-hand side of (\ref{3.14}) by $D(\gamma,\eta)$, we arrive to the equation
\begin{equation}  \label{3.15}
(KF)(\gamma,\eta)=D(\gamma,\eta)\quad (\eta\in\gamma^\bot).
\end{equation}

Let us distinguish the trace free part of the tensor $F$, i.e., represent it in the form
$$
F=\tilde F +\frac{1}{3}(\mbox{tr}\,F)h,\quad\mbox{where}\quad\mbox{tr}\,\tilde F=0.
$$
By Theorem 6.6.2 of \cite{mb}, the trace free tensor field $\tilde F$ can be uniquely recovered from $Kf$ if the CNTM $(M,h)$ satisfies some curvature condition. Moreover, the stability estimate
$$
\|\tilde F\|_{L^2}\leq C \|KF\|_{H^1}
$$
holds with some constant $C$ independent of $F$.

On assuming $\tilde F$ has been recovered, we can calculate $I\tilde F$ and $J\tilde F$. Then
$$
\begin{aligned}
(JF)(\gamma,\eta)&+|\eta|^2(IF)(\gamma)\\
&=\frac{1}{3}(J((\mbox{tr}\,F)h))(\gamma,\eta)+\frac{|\eta|^2}{3}(I((\mbox{tr}\,F)h))(\gamma)
+(J\tilde F)(\gamma,\eta)+|\eta|^2(I\tilde F)(\gamma)\\
&=\frac{|\eta|^2}{3}
(I((\mbox{tr}\,F)h))(\gamma)+\frac{|\eta|^2}{3}(I((\mbox{tr}\,F)h))(\gamma)+(J\tilde
F)(\gamma,\eta)+|\eta|^2(I\tilde F)(\gamma).
\end{aligned}
$$
Substituting this value into (\ref{3.12}), we obtain the equation
\begin{equation}  \label{3.16}
(I((a+\frac{2}{3})(\mbox{tr}\,F))(\gamma)=\frac{1}{|\eta|^2}\Big((Lf)(\gamma,\eta)-(J\tilde F)(\gamma,\eta)\Big)-(I\tilde F)(\gamma).
\end{equation}
The left-hand side of this equation is the ray transform of the scalar function $(a+\frac{2}{3})(\mbox{tr}\,F)$ while the right-hand side is known. By the way, the right-hand side must be independent of $\eta\in\gamma^\bot$ as follows from the equation.

By Mukhometov's theorem \cite{Mu}, the function $(a+\frac{2}{3})(\mbox{tr}\,F)$ can be uniquely recovered from the ray transform $I((a+\frac{2}{3})(\mbox{tr}\,F))$ if $(M,h)$ is a simple manifold, see Section 1.1 of \cite{mb} for the definition of a simple manifold. Thus, the trace $\mbox{tr}\,F$ can be recovered under the additional assumption that the function $3a+2$ does not vanish. Since the trace free part $\tilde F$ has been already recovered, this gives the uniqueness statement for a solution to equation (\ref{3.12}). The corresponding stability estimate can be also obtained.

\bigskip

Finally, let us discuss the case of constant material parameters
$\lambda,\mu,\rho,\nu_1,\dots,\nu_4$. In this case, to
solve equation (\ref{3.12}), we do not need to measure ray integrals
$(Lf)(\gamma,\eta)$ for all lines $\gamma$ of ${\mathbb R}^3$.
Indeed, as is proved in \cite{LS}, three families of lines are
sufficient to recover the trace free part $\tilde F$ from the data
$KF$, each family consists of all lines parallel to a coordinate
plane. To solve equation (\ref{3.16}), it suffices to know the
right-hand side for all lines $\gamma$ parallel to a plane.

\end{document}